\documentclass{article}

\usepackage[utf8]{inputenc}
\usepackage{amsfonts}
\usepackage{epigraph}

\usepackage{xcolor}

\usepackage{amsfonts}
\usepackage{amssymb}
\usepackage{amsmath}
\usepackage{amssymb}
\usepackage{float}
\usepackage{amsthm}
\usepackage[pdftex]{graphicx}

\usepackage{fullpage}
\usepackage{wrapfig}

\newtheorem{theorem}{Theorem}
\newtheorem{corollary}{Corollary}
\newtheorem{lemma}{Lemma}

\def\lc{\left\lceil}   
\def\rc{\right\rceil}

\DeclareMathOperator{\sign}{sign}

\title{On the chromatic numbers of Johnson type graphs}
\author{Danila Cherkashin}

\begin{document}

\maketitle

\begin{abstract}
A Johnson type graph $J_{\pm}(n,k,t)$ is a graph whose vertex set consists of vectors from $\{-1,0,1\}^n$ of the length $\sqrt{k}$ and
edges connect vertices with scalar product $t$.
The paper determines the order of growth of the chromatic numbers of graphs $J_\pm(n,2,-1)$ and $J_\pm(n,3,-1)$ (logarithmic on $n$), and also $J_\pm(n,3,-2)$ (double logarithmic on $n$).
\end{abstract}

\section{Notation}

Recall that the \textit{chromatic number} of a graph $G$ is the minimum number of colors $\chi(G)$ required to color the vertices of the graph so that any edge has ends of different colors. An \textit{independent set} in a graph is a set of vertices, no two of which are adjacent.
The \textit{independence number} of a graph $G$ is the size of its largest independent set; the independence number is denoted by $\alpha(G)$.

A \textit{generalized Johnson graph} $J(n,k,t)$ is defined as the graph whose vertices are the vectors from the hypercube ${0,1}^n$ with exactly $k$ ones,
and whose edges connect pairs of vertices with scalar product $t$ (that is, $J(n,k,t)$ is nonempty if $k<n$ and $2k-n \leq t < k$).
\textit{Generalized Kneser graphs} $K(n,k,t)$ have the same vertex set, but edges connect pairs of vertices with scalar product at most $t$.
Johnson-type graphs, defined in the next paragraph, are a natural generalization of Johnson graphs.

A \textit{Johnson-type graph} $J_\pm(n,k,t)$ is defined as the graph whose vertices are the vectors from the set $\{0,\pm 1\}^n$ of length $\sqrt{k}$, and whose edges connect pairs of vectors with scalar product $t$.

Note that the vertices of $J_\pm(n,k,t)$ are naturally embedded in $\mathbb{R}^n$, and the distance between adjacent vertices is $\sqrt{2(k-t)}$.
Thus, $J_\pm(n,k,t)$ is a distance graph, meaning that its edges connect pairs of vertices at a fixed distance.
We assume that the coordinates are numbered, in other words, the set of coordinates is $[n] = \{1,\dots,n\}$.
The \textit{support} of a vertex is the set of its nonzero coordinates.
For $k=2$, we use the notation $a^ib^j$ for a vertex with support ${a,b}$ and signs $i,j \in \{+,-\}$ in coordinates $a$ and $b$, respectively; similar notation is used for $k=3$.

An \textit{automorphism} of a graph is a bijection of its vertex set onto itself that preserves adjacency.
A graph is called vertex-transitive if for any vertices $u$ and $v$ there exists an automorphism of the graph mapping $u$ to $v$.

Finally, let $m(a,b)$ denote the position of the most significant differing bit in the binary representations of the numbers $a$ and $b$ (the digits are numbered starting from one).

\section{Introduction}

In this note we determine the growth order of the chromatic numbers of the graphs $J_\pm(n,2,-1)$ and $J_\pm(n,3,-1)$ (logarithmic in $n$), as well as of $J_\pm(n,3,-2)$ (double logarithmic in $n$).

A survey of results on the chromatic numbers of generalized Johnson and Kneser graphs for fixed $k$ and $t$ can be found in~\cite{cherkashin2023independence}.

In~\cite{cherkashin2023independence} exact values of the independence numbers of the graphs $J_\pm(n,k,t)$ were obtained for fixed $k$, sufficiently large $n$, and zero or odd negative $t$.
Asymptotics of the independence numbers of $J_\pm(n,k,t)$ for fixed $k$, sufficiently large $n$, and negative $t$ were established in~\cite{frankl2016intersection} (see also the corrected version~\cite{frankl2019correction}).
Exact values of the independence numbers for even negative $t$ remain unknown, starting already from the case $J_\pm(n,3,-2)$.

Let us now turn to chromatic numbers. Suppose $k$ is fixed and $t$ is negative.
Then all vectors with nonnegative coordinates form an independent set $I$ with vertex density $\tfrac{1}{2^k}$.
Thus, for any Johnson-type graph $G = J_\pm(n,k,t)$ the classical inequality
\begin{equation}
\chi(G) \geq \frac{|V(G)|}{\alpha(G)}
\label{1}
\end{equation}
yields a lower bound on the chromatic number not exceeding $2^k$.
Later we will see that this inequality is often very far from sharp.

On the other hand, all graphs considered in this paper are vertex-transitive.
Therefore, for $G = J_\pm(n,k,t)$ the following inequality, valid for all vertex-transitive graphs~\cite{lovasz1975ratio}, holds:
\begin{equation}
\chi(G) \leq (1 + \ln \alpha(G)) \frac{|V(G)|}{\alpha(G)} \leq C(k) \log_2 n,
\label{2}
\end{equation}
where $C(k)$ is a constant depending only on $k$.
We will see that sometimes this estimate gives the correct order of growth in $n$, in particular for $J_\pm(n,3,-1)$.

\section{Results}

As a warm-up, let us properly color $J_\pm(n,2,-1)$ in $2 \lc \log_2 n \rc + 2$ colors. 
Let the first and second colors be assigned to the vertices with all nonnegative and all nonpositive coordinates, respectively. 
The remaining vertices are of the form $a^+b^-$. 
We color a vertex $a^+b^-$ with color $m(a,b)$ if in the bit $m(a,b)$ the number $a$ has $1$ and the number $b$, respectively, $0$; 
similarly, the vertex $a^-b^+$ is colored with $-m(a,b)$. 
It is easy to see that all vertices are colored and every edge connects vertices of different colors. 
Thus we obtain exactly $2 \lc \log_2 n \rc + 2$ colors.

The following theorem shows that the asymptotic behavior of the chromatic number is roughly twice smaller than in this example.

\begin{theorem} For all $n \geq 2$ the following inequalities hold
\[
\log_2 n \leq \chi(J_\pm(n,2,-1)) \leq \log_2 n + \left(\tfrac{1}{2}+o(1)\right)\log_2 \log_2 n.
\]
\end{theorem}

\begin{proof}
We begin with the lower bound. 
Consider the subset of vertices containing coordinates of different signs; denote it by $V_\pm$. 
Suppose we cover $V_\pm$ with independent sets $I_1,\dots,I_q$. 
Take a set $I_j$. 
If for some coordinate $a$ the vertices of $I_j$ take both signs, then every vertex of $I_j$ whose support contains $a$ is of the form $a^+b^-$ or $a^-b^+$ for some coordinate $b$. 
Thus, on coordinate $b$ the vertices of $I_j$ also take both signs. 
We call such coordinates \textit{diverse}, and the remaining coordinates \textit{positive} or \textit{negative}, respectively.

Define an auxiliary graph $G_j$, whose vertices are the coordinates, and where an edge connects a pair of coordinates if this pair forms the support of some vertex from $I_j$. 
Note that $G_j$ is bipartite: indeed, the support either contains a positive and a negative coordinate (since we only consider vertices from $V_\pm$), or two diverse ones; as we showed above, $G_j$ restricted to the diverse coordinates is a matching.

Since $I_1,\dots,I_q$ cover $V_\pm$, the graphs $G_j$ cover $K_{[n]}$, the complete graph on the coordinate set. 
It is well known that such a covering requires at least $\log_2 n$ bipartite subgraphs. 
Indeed, if $F_i := G_1 \cup \dots \cup G_i$, then $\alpha(F_{i+1}) \geq \alpha(F_i)/2$; on the other hand, $\alpha(K_{[n]}) = 1$.

Now let us move to the example. 
We will color $J_\pm(n,2,-1)$ with $2m+2$ colors for $n \leq \binom{2m+1}{m}$. 
Associate to each coordinate an $m$-element subset of $[2m+1]$; denote this assignment by $f$. 
For $1 \leq i \leq 2m+1$, the color $I_i$ consists of the vertices $a^+b^+$ for which $i \in f(a), f(b)$, 
the vertices $a^+b^-$ for which $i \in f(a)$ and $i \notin f(b)$, 
and the vertices $a^-b^-$ for which $i \notin f(a),f(b)$. 
Note that all vertices of the form $a^+b^-$ are covered by these colors: indeed, for any two $m$-element subsets $f(a), f(b)$ there exists an element $i \in f(a)\setminus f(b)$. 
Similarly, all vertices of the form $a^-b^-$ are covered, since for any two $m$-element subsets of a $(2m+1)$-element set there exists an element in their common complement. 
The last color, numbered $2m+2$, contains all vertices of the form $a^+b^+$.  

Since $J_\pm(n_1,2,-1)$ is a subgraph of $J_\pm(n_2,2,-1)$ whenever $n_1 < n_2$, it remains to check that the inequality 
\[
n \geq \binom{2m-1}{m-1}
\]
implies
\[
2m+2 \leq \log_2 n + (1+o(1)) \tfrac{1}{2}\log_2 \log_2 n.
\]
This is indeed the case, since
\[
\binom{2m-1}{m-1} = \tfrac{1}{2}\binom{2m}{m} = \Omega\!\left(\frac{4^m}{\sqrt{m}}\right).
\]
\end{proof}

Let $I$ be an independent set in the graph $J_\pm(n,3,-1)$. 
We call a coordinate $i \in [n]$ \textit{diverse} if the vertices from $I$ take both nonzero values on $i$. 
We call a vertex $v \in I$ \textit{special} if the support of $v$ contains a diverse coordinate.

\begin{lemma}
Let $I$ be an independent set in the graph $J_\pm(n,3,-1)$ for which $t$ coordinates are diverse. Then
\[
|I| \leq 8t(n-2) + \binom{n-t}{3}.
\]
\end{lemma}

\begin{proof}
Consider a diverse coordinate $i$. 
Let $I_i$ be the subset of vertices of $I$ whose support contains $i$.  
Then $I_i$ contains a vertex $v_+$ with support $\{i,a,b\}$ and a $+$ sign on $i$, and also a vertex $v_-$ with support $\{i,c,d\}$ and a $-$ sign on $i$ (the sets $\{a,b\}$ and $\{c,d\}$ must intersect).  
Then any vertex $v \in I_i$ with a $+$ sign on $i$ must intersect $\{c,d\}$ (otherwise $v_+$ and $v_-$ would be adjacent, contradicting the independence of $I$); similarly, any $v$ with a $-$ sign on $i$ must intersect $\{a,b\}$.  
Thus
\[
|I_i| \leq 8 (n-2) \quad \quad \mbox{and} \quad \quad \bigcup_{1\leq i \leq t} I_i \leq 8 t(n-2).
\]
We have counted all vertices having at least one diverse coordinate.  
The remaining vertices have their support on the $n-t$ non-diverse coordinates.  
Their number is at most $\binom{n-t}{3}$.
\end{proof}

\begin{corollary}
Let $I$ be an independent set in the graph $J_\pm(n,3,-1)$. 
Then the number of special vertices does not exceed $cn^2$.
\end{corollary}

\begin{theorem}
For some positive constants $c,C$ and any $n > 3$ the following inequalities hold:
\[
c\log_2 n \leq \chi (J_\pm [n,3,-1] ) \leq C \log_2 n.
\]
\end{theorem}

\begin{proof}
The upper bound follows from the fact that the graph $J_\pm(n,3,-1)$ is vertex-transitive, as was shown in the introduction.

Let us turn to the lower bound.  
Consider a partition of the vertex set $V$ of the graph $J_\pm(n,3,-1)$ into independent sets $I_1, \dots, I_k$.  
Let $V_\pm \subset V$ be the set of vertices that have coordinates of both signs; the number of such vertices is $6\binom{n}{3}$.  
Note that for each pair of coordinates $(i,j)$ there are $4(n-2)$ vertices from $V_\pm$ in which $i$ and $j$ have different signs.  
We say that a set of vertices \textit{separates} a pair of coordinates $(i,j)$ if the set contains all these $4(n-2)$ vertices.  
Then, by Corollary~1, the special vertices of all independent sets in the union separate $O(kn)$ pairs of coordinates.

The completion of the proof is analogous to Theorem~1.  
Let the independent sets $I_1, \dots, I_q$ cover the graph $G$.  
Consider the bipartite graph $F_i$ on the set of coordinates, whose parts are the uniform coordinates of $I_i$ with positive and negative signs.  
As was shown in the proof of Theorem~1, the graph $F := F_1 \cup \dots \cup F_q$ has an independent set of size at least $n/2^q$.  
If $q < \tfrac{1}{3}\log_2 n$, then
\[
\alpha(F) \geq n^{2/3}.
\]
Then all pairs of coordinates on an independent set of size $n^{2/3}$ must be separated by diverse coordinates;  
there are asymptotically more such pairs than $kn$ --- a contradiction.
\end{proof}

\begin{theorem} For all $n \geq 3$ the following inequalities hold:
\[
\lc \log_2 \lc \log_2 n \rc \rc \leq \chi (J_\pm [n,3,-2] ) \leq 4 \lc \log_2 \lc \log_2 n \rc \rc + 6.
\]
\label{mainTH}
\end{theorem}

\begin{proof}
We begin with the lower bound.  
Suppose that we cover all vertices of the graph by independent sets $I_1, \dots, I_q$.  
Fix an order on the coordinates $[n]$ and consider only the subset of vertices $V_{alt}$ in which the signs alternate, that is, vertices with support $\{a,b,c\}$, where $a < b < c$, are of the form $a^+b^-c^+$ and $a^-b^+c^-$.

Consider the auxiliary graph $H$, whose vertices are unordered pairs of coordinates, and whose edges connect pairs of the form $\{a,b\}$ and $\{b,c\}$, where $a < b < c$.  
Then for each edge of $H$ on the union of vertices (as a support) there are two vertices of the graph $G$ from $V_{alt}$; conversely, the support of any vertex from $V_{alt}$ is uniquely obtained as the union of the vertices corresponding to an edge of $H$.

\begin{lemma}
For each $1 \leq j \leq q$ the edges corresponding to the supports of $I_j$ form a bipartite subgraph in $H$.
\end{lemma}

We denote this subgraph by $H_j$.

\begin{proof}
Assign two labels to each vertex from $I_j \cap V_{alt}$: if the vertex has the form $a^+b^-c^+$ (with respect to the order on $[n]$), then the coordinate pairs $\{a,b\}$ and $\{b,c\}$ receive the labels $L$ and $R$, respectively; if it has the form $a^-b^+c^-$, then vice versa.  
Suppose some coordinate pair $\{d,e\}$ receives both label $L$ and label $R$ from vertices $v_1$ and $v_2$.  
Then the supports of $v_1$ and $v_2$ coincide, otherwise their scalar product equals $-2$, contradicting the independence of $I_j$.  
But then the support of any other vertex from $I_j \cap V_{alt}$ does not contain the pair $\{d,e\}$, otherwise an edge would appear either with $v_1$ or with $v_2$.

Now put all vertices of $H$ with only label $L$ in one part, and those with only label $R$ in the other.  
The subgraph on these vertices is bipartite.  
Vertices with both labels, as shown above, are isolated, so adding them preserves bipartiteness.
\end{proof}

Let the complete bipartite subgraphs on the parts of the subgraphs $H_i$ partition $V(H)$ into independent (in $H$) sets $J_1,\dots,J_w$.

\begin{lemma}
Let $J$ be an independent set in the graph $H$.  
Then there exists a partition $[n] = B \sqcup E$ such that for any vertex $\{b,e\} \in J$ one has $b \in B$, $e \in E$, and $b < e$.
\end{lemma}

\begin{proof}
No coordinate can be the first in a vertex $j_1 \in J$ and the second in a vertex $j_2 \in J$, since $J$ is an independent set and such $j_1$ and $j_2$ would form an edge.  
This allows us to define $B$ as the set of first coordinates of vertices $j \in J$, and $E$ as $[n] \setminus B$.
\end{proof}

For each graph $J_i$ consider the partition $B_i \sqcup E_i$ from Lemma~2.  
Suppose that some pair of coordinates does not lie in different parts of any partition $B_i \sqcup E_i$.  
Then the corresponding vertex of $H$ does not belong to the union of the independent sets $J_1,\dots,J_w$, a contradiction.  
Consider the union $F$ of the bipartite graphs on the parts $B_i,E_i$ for $1 \leq i \leq w$.  
All parts of $F$ have size $1$, hence $w \geq \lc \log_2 n \rc$, analogously to the corresponding part of the proof of Theorem~1.  

It follows that the subgraphs $H_1, \dots, H_q$ partition $V(H)$ into at least $\lc \log_2 n \rc$ sets, that is, the number of subgraphs is at least $\lc \log_2 \lc \log_2 n \rc \rc$, which completes the proof of the lower bound.

Let us now turn to the upper bound and exhibit a coloring of the graph in $4 \lc \log_2 \lc \log_2 n \rc \rc + 6$ colors.  

First, color all vertices of the form $a^+b^-c^+$, where $a < b < c$, in $2 \lc \log_2 \lc \log_2 n \rc \rc$ colors.  
Assign to such a vertex the color $\sign (m(a,b) - m(b,c)) \cdot m( m(a,b), m(b,c) )$.  
Note that for natural numbers $x < y < z$ one has $m(x,y) \neq m(y,z)$.  
Indeed, from $x < y$ it follows that in the bit $m(x,y)$ the number $x$ takes the value $0$ while $y$ takes $1$; a similar argument for $y$ and $z$ yields a contradiction.

Suppose there exists a pair of vertices of color $j$ with supports $\{a,b,c\}$, where $a < b < c$, and $\{b,c,d\}$ with scalar product $-2$.  
Then, since the first vertex is of the form $a^+b^-c^+$, the second vertex has signs $b^+c^-$, hence the coordinates are ordered $a < b < c < d$.  
Since the vertices are monochromatic, the expressions $m(a,b) - m(b,c)$ and $m(b,c) - m(c,d)$ have the same sign.  
If $m(a,b) < m(b,c) < m(c,d)$, then $m( m(a,b), m(b,c) ) = m( m(b,c), m(c,d) )$ implies that $m(b,c)$ has bit $|j|$ equal to $1$ on the one hand and $0$ on the other; a contradiction.  
The case $m(a,b) > m(b,c) > m(c,d)$ is handled similarly.

Analogously, one can color the vertices of the form $a^-b^+c^-$ ($a < b < c$) in $2 \lc \log_2 \lc \log_2 n \rc \rc$ colors.

Finally, the last six colors consist of the vertices of the forms $a^+b^+c^+$, $a^+b^+c^-$, $a^+b^-c^-$, $a^-b^+c^+$, $a^-b^-c^+$ and $a^-b^-c^-$, respectively, where $a < b < c$.  
A direct check shows that under this coloring there are no monochromatic edges.
\end{proof}

\section{Discussion}

\paragraph{A contest version} of of Theorem~\ref{mainTH} was proposed at the student competition IMC 2022. I am grateful to John Jayne  for his strong recommendation not to attend this event,  which turned out to be very timely.

\begin{quote}
\textbf{Problem 4.}  Let $n > 3$ be an integer. Let $\Omega$ be the set of all triples of distinct elements of
$\{1, 2, \dots , n\}$. Let $m$ denote the minimal number of colours which suffice to colour $\Omega$ so that
whenever $1 \leq a < b < c < d \leq n$, the triples $\{a, b, c\}$ and $\{b, c, d\}$ have different colours. Prove
that
\[
\frac{1}{100} \log \log n \leq m \leq 100 \log \log n.
\]
\end{quote}

\paragraph{Further questions.} Determining the chromatic numbers of the graphs $J_\pm (n,k,t)$ and $K_\pm (n,k,t)$ in the general case seems to be a hopeless task.  
In the case of fixed $k$ and $t$, the methods mentioned in the introduction (inequalities~\eqref{1} and~\eqref{2}) give a logarithmic in $n$ gap between the upper and lower bounds.  
We have shown that for negative $t$ the chromatic number may differ in order from both of these bounds.  
For nonnegative $t$ no progress is known to the author.

\paragraph{Acknowledgements.} The work was supported by grant 21-11-00040 of the Russian Science Foundation.
The author thanks Fedor Petrov, Alexey Gordeev and Pavel Prozorov for their interest in the work and significant assistance in presenting the results.

\bibliographystyle{plain}
\bibliography{main}

\begin{thebibliography}{1}

\bibitem{cherkashin2023independence}
Danila Cherkashin and Sergei Kiselev.
\newblock Independence numbers of {J}ohnson-type graphs.
\newblock {\em Bulletin of the Brazilian Mathematical Society, New Series},
  54(3):30, 2023.

\bibitem{frankl2016intersection}
P.~Frankl and A.~Kupavskii.
\newblock Intersection theorems for $\{0,\pm 1\}$-vectors and
  $s$-cross-intersecting families.
\newblock {\em Moscow Journal of Combinatorics and Number Theory},
  2(7):91--109, 2017.

\bibitem{frankl2019correction}
P.~Frankl and A.~Kupavskii.
\newblock Correction to the article {I}ntersection theorems for
  (0,$\pm$1)-vectors and $s$-cross-intersecting families.
\newblock {\em Moscow Journal of Combinatorics and Number Theory},
  8(4):389--391, 2019.

\bibitem{lovasz1975ratio}
L.~Lov{\'a}sz.
\newblock On the ratio of optimal integral and fractional covers.
\newblock {\em Discrete Mathematics}, 13(4):383--390, 1975.

\end{thebibliography}

\end{document}